\documentclass[11pt]{amsart}
\usepackage{amssymb,amsfonts,amsthm,amsmath}
\usepackage[english]{babel}
\usepackage[all,cmtip]{xy}

\newtheorem{theorem}{Theorem}[section]
\newtheorem{lemma}[theorem]{Lemma}
\newtheorem{corollary}[theorem]{Corollary}
\newtheorem{definition}[theorem]{Definition}
\newtheorem{proposition}[theorem]{Proposition}
\newtheorem{remark}[theorem]{Remark}

\newcommand{\field}[1]{\mathbb{#1}}

\newcommand{\bB}{{\bf B}}
\newcommand{\bdr}{{\partial}}

\newcommand{\bo}{{\bf 0}}
\newcommand{\bu}{{\bf u}}

\newcommand{\calH}{\mathcal{H}}

\newcommand{\calM}{\mathcal{M}}
\newcommand{\calU}{\mathcal{U}}
\newcommand{\calV}{\mathcal{V}}

\newcommand{\clos}{{\rm clos}}
\newcommand{\cone}{{\rm Cone}}

\newcommand{\din}{{\rm d_{\rm inner}}}
\newcommand{\dist}{{\rm dist}}
\newcommand{\dou}{{\rm d_{\rm outer}}}

\newcommand{\intr}{{\rm int}}

\newcommand{\ns}{{\rm ns}}

\newcommand{\Q}{\field{Q}}
\newcommand{\R}{\field{R}}
\newcommand{\Rn}{\R^n}
\newcommand{\smp}{{\rm smp}}
\newcommand{\Sns}{{\bf S}_\bo^{\rm ns}}
\newcommand{\Sos}{{\bf S}_\bo^{\rm smp}}
\newcommand{\So}{{\bf S}_\bo}

\newcommand{\Sr}{{\bf S}}
\newcommand{\Srn}{\Sr^{n-1}}
\newcommand{\supp}{{\rm supp}}
\newcommand{\thin}{{\rm Thin}}
\newcommand{\ud}{{\rm d}}
\newcommand{\ve}{{\varepsilon}}

\newcommand{\vol}{{\rm vol}}
\newcommand{\bs}{ {\tiny $\blacksquare$} \\}
\newcommand{\Z}{\field{Z}}
\numberwithin{equation}{section}

\title[Collapsing topology of isolated singularities]{Collapsing topology of isolated singularities}

\author[L. Birbrair]{Lev Birbrair$^1$}\thanks{$^1$Research supported under CNPq grant no 300575/2010-6}
\address{Departamento de Matem\'atica, Universidade Federal do Cear\'a
(UFC), Campus do Picici, Bloco 914, Cep. 60455-760. Fortaleza-Ce,
Brasil} \email{birb@ufc.br}
\author[A. Fernandes]{Alexandre Fernandes$^2$}\thanks{$^2$Research supported under CNPq grant no 302998/2011-0}
\address{Departamento de Matem\'atica, Universidade Federal do Cear\'a
(UFC), Campus do Picici, Bloco 914, Cep. 60455-760. Fortaleza-Ce,
Brasil} \email{alexandre.fernandes@ufc.br}
\author[V. Grandjean]{Vincent Grandjean$^3$}\thanks{$^3$Research supported by CNPq grant no 150555/2011-3}
\address{Departamento de Matem\'atica, Universidade Federal do Cear\'a
(UFC), Campus do Picici, Bloco 914, Cep. 60455-760. Fortaleza-Ce,
Brasil}
\email{vgrandje@fields.utoronto.ca}
\thanks{}

\date{\today}

\begin{document}
\begin{abstract}
We proof here the existence of a topological thick and thin decomposition
of any closed definable thick isolated singularity germ in the spirit of the 
recently discovered metric thick and thin decomposition of complex normal surface singularities
of \cite{BNP}. Our thin zone catches exactly the homology of the family of the links
collapsing faster than linearly. Simultaneously we introduce
a class of rigid homeomorphisms more general than bi-Lipschitz ones,    
which map the topological thin zone onto the topological thin zone of its image. 
As a consequence of this point of view for the class of singularities we consider 
we exhibit an equivalent description of the notion of separating sets in terms of this fast 
contracting homology. 
\end{abstract}
\maketitle
\tableofcontents

\section{Introduction}
A subset of a Riemannian manifold carries two natural metrics: the \em outer metric \em where the 
distance between points 
is measured in the ambient space, and the \em inner metric \em where the infimum of the lengths
of rectifiable curves joining the points considered realizes the distance. 
Although the underlying topological structures are homeomorphic, 
the corresponding metric spaces may not be equivalent.

\smallskip 
Real tame or complex analytic subsets share the same local structure, namely they are topological cones over 
their link \cite{Loj,Mil,vdDMi}. (The adjective \em tame \em has to be understood as 
\em topologically tame, \em as meant in Grothendieck's \em Esquisse d'un programme\em{).}
Since such subsets are always (locally) embedded in a real 
or complex Euclidean space, they come equipped with the inner metric and with the outer metric. 
Although the local inner and outer geometric structure of real tame and complex subsets 
are more delicate to understand than their topology, some tools are available to do so. 
Following the works of the first two named authors \cite{Fe1,Bi2} for curves (see also \cite{PhTe}) 
and surfaces definable in an o-minimal structure expanding the field of real numbers, and of Grieser 
\cite{Gri} for real analytic surfaces with isolated singularities, the notion of metric type introduced 
by Valette \cite{Va1,Va2} gives a thorough description of the local metric combinatorics of the 
corresponding subsets. The topology of the family of the links - the family of intersections of a representative 
of the germ with the Euclidean spheres centered at the origin and of small radii - of tame isolated singularities 
may be rich, but its metric behavior when the radii tends to $0$ has still a long way to go before being fully 
understood. Once more some tools are available to deals with such a problem: Birbrair-Brasselet metric homology 
\cite{BiBr1,BiBr2} and Valette vanishing homology \cite{Va3} provide some insights into 
this topic. But more is needed. 

The bottom line of the investigation from a metric point of view (inner or outer)
is to understand the locally conic geometry of the class of germs and then proceed to higher degenerate metric 
types. Unlike tame real sets, complex analytic set germs have tangent cones of maximal dimension, consequently 
most of the (inner and outer) metric type is of conic nature. The principal reason to insist 
on the locally conic locus of a given singularity set-germ is - 
as it happens for cones - that the collapsing of the topology of the family of links 
is not \em faster than linearly, \em that is not faster than the velocity of the parameter (radius) 
of the family tending to $0$ (see \cite{BiBr1,BiBr2,Va2} to give meanings to this expression). 

Prior to the paper by the two first named authors \cite{BF3} was a hazy belief that germs of 
complex analytic sets could be conic. They found the first example of a complex isolated singularity 
(a Pham-Brieskorn surface singularity) which cannot be conic at its singular point \cite{BF3} and did so exhibiting 
a so called \em separating set \em which produces a homological obstruction in the family of links. 
Later, joined by Neumann, they provided more examples of non-conical isolated complex 
surface singularities \cite{BF3,BFN1,BFN3}. In their very recent work \cite{BNP}, the first named author, 
Neumann and Pichon completely describe the inner metric structure of complex normal surface singularities. 
They exhibit a (\em metric) \em decomposition of the surface germ into thick pieces and thin pieces:
The thin pieces catch exactly all the topology of the family of links collapsing onto the origin faster than   
linearly. These authors do not only give the 
precise rates of each collapsing, but they also describe very precisely the structures 
corresponding to each rate and how they are organized relatively to each other by means of a rather elementary
combinatorial object. A second remarkable feat of this decomposition is that the complement of any conical neighborhood 
of the thin zone is an inner-metric cone. Their decomposition is minimal for these two properties. 
This work rely on the very rich amount of data available for surface complex singularities and topology 
of $3$-folds, and how all these different points of view fit together.

\medskip
Although the aforementioned works relate to the Lipschitz Geometry of 
tame subsets, the point of view of our paper is slightly more general.
Inspired by the quoted paper \cite{BNP}, our starting point was 
to look for the existence of a real tame avatar of a thick-and-thin-like decomposition for 
any tame isolated singularity germ. Can such singularity germs have a \em topological \em thick and thin 
decomposition, namely: \em Does there exist a decomposition into thick zones and thin zones such that the 
latter one catches the topology of the family of the links collapsing faster than linearly? \em 
\\
As we will see our slightly more general point of view on the topology of the family of the links will 
bring tools to investigate the local metric conicalness. Note also that describing the metric 
structure of the thin zone in our general framework is an extremely arduous task, since there does
not exist any practically usable quantified complexity of such structure, although such formal data exist in 
the polynomially bounded case \cite{Va2}. 

\smallskip
The present paper proposes a solution to the problem of the existence of a topological 
thick and thin zone for thick tame isolated singularities. It is organized as follows.

In order to catch the fast contracting topology of the link, we introduce the notion of \em fast contracting 
cycles, \em which are cycles supported in the link and contracted in the subset by a tame chain whose support 
is thin and contains the origin (Definition \ref{def:admis-van-hom}). Attached to this classification problem comes 
the right class of mappings that we call \em morphisms, \em which are tame continuous mappings which extend as
tame and continuous mappings after spherical blowings-up at the source and the target (Definition \ref{def:morphism}). 
The corresponding homeomorphisms will be called \em isomorphisms. \em
The idea of using spherical blowings-up is completely natural when it comes to deal 
with the tangent directions at the center of the blowing-up, 
since the tangent link is embedded as the boundary of the strict transform (defined as the intersection 
of the strict transform with the ambient exceptional divisor). 
It turns out that the fast contracting cycles form a sub-complex of the usual homology of the link and that 
any morphism induces a homomorphism between the corresponding subgroups (Proposition \ref{prop:homom-fc-homol}).
Thanks to these simple notions we can weaken the notion of \em normal embedding model \em of a (tame) subset
\cite{BiMo} into the notion of \em simply embedded model \em which is any isomorphic image of a normally embedded model
(Definition \ref{def:SE-mod}). We then find out necessary 
and sufficient conditions on a simply embedded model (of a tame closed thick singularity) to be a cone over a 
topological tame submanifold (Lemma \ref{lem:loc-conic-homeo-cone}). 


\smallskip
In order to find this topological thick and thin composition for tame closed thick isolated singularity 
our proof rely mostly on a weak version of Federer-Flemming Theorem on the triviality of cycles with small 
volumes \cite{FF}. Namely we prove  

\medskip\noindent
{\bf Theorem \ref{thm:FF-polyhedra}.} \em Let $\Gamma$ be a compact polyhedron defined as the union of the 
simplices of a finite set $T$. Let $\xi$ be a definable  $k$-cycle in $\Gamma$ such that for any positive 
real number $\ve$ small enough there exists a homologous definable $k$-cycle $\xi_{\ve}$ in $\Gamma$ whose 
$k$-dimensional volume $\vol_k(\supp(\xi_{\ve}))$ is smaller than $\ve$. Then the cycle $\xi$ bounds a 
definable $(k+1)$-chain of $\Gamma$. \em

\medskip
As expected consequence of Theorem \ref{thm:FF-polyhedra}, cones and their isomorphism classes cannot carry any fast 
contracting homology (Corollary \ref{cor:cone-no-hom}).

\smallskip
Having identified the locus of \em non-simple tangent directions \em $\Sns X_1$ as 
the subset of points in the tangent link of a simply embedded model $X_1$ of a closed tame thick isolated 
singularity where the fast contracting cycles must collapse (Definition \ref{def:simple-dir}), 
and showed that it is tame closed and of codimension at least $1$ in the tangent 
link of $X_1$, we can prove the first version of our \em topological thick and thin decomposition: \em  

\medskip\noindent
{\bf Theorem \ref{thm:main}.} \em
Let $X_1$ be a simply embedded model of a closed definable isolated singularity of some Euclidean
space, thick at the origin  $\bo$ and with connected link at the origin. 
For any positive $\ve$ small enough, the inclusion mapping induces the following exact sequence 
\begin{equation*}
\calH_\bullet (X_1\cap \cone(\Sns X_1, \ve),\bo) \to \calH_\bullet (X_1,\bo) \to 0.
\end{equation*}
where $\calH_\bullet (X_1,\bo)$ is the sub-complex corresponding to the fast contracting homology classes.
\em

\medskip
The proof of Theorem \ref{thm:main} consists of finding a definable 
triangulation, which exists, in the simplices of which we can apply Theorem \ref{thm:FF-polyhedra} 
whenever possible.

Since the existence of a topological thick and thin decomposition is now established, we define 
the topological thin zone in a systematic way (Definition \ref{def:thin-zone}) as the image, by the 
spherical blowing-down mapping, of any \em tame neighborhood \em (Definition \ref{def:def-nghb}), in the strict 
transform, of the embedding of the non-simple point locus in the boundary of the 
strict transform as the subset $\Sns X_1\times 0$. We check that this definition makes sense (Proposition 
\ref{prop:defin-neighb}). Our main result can be rephrased as

\medskip\noindent
{\bf Theorem \ref{thm:main-bis}.} \em 
Let $X_1$ be a simply embedded model of a closed definable isolated singularity of some Euclidean
space, thick at the origin and with connected link. 
The inclusion mapping induces the following exact sequence 
\begin{equation*}
\calH_\bullet (\thin (X_1),\bo) \to \calH_\bullet (X_1,\bo) \to 0.
\end{equation*}
where $\thin(X_1)$ is the topological thin zone of the germ $(X_1,\bo)$. 
\em
%
%
%
 
\smallskip 
Eventually, we present a few examples to illustrate our 
point of view. All the (so far) known obstructions to local inner metric conicalness,
namely \em fast loops, separating sets \em and \em choking horns, \em can be very well described by the 
fast contracting homology of the topological thin zone. These notions are not only bi-Lipschitz invariant but
also isomorphic invariant when it comes to simply embedded models of tame thick isolated singularities.
We insist particularly on the case of separating sets.
Condition (SC), introduced in Definition \ref{def:cond-SC}, holds true for 
simply embedded model of closed tame thick isolated singularities if and only if there exists 
a separating set. 
We end the paper by bringing to light the relation of the existence of separating sets
with the fast contracting homology of the toplogical thin zone:

\medskip\noindent
{\bf Theorem \ref{thm:sep-set-pure-dim}.}
\em Let $X_1$ be a simply embedded model of a closed definable isolated singularity of some Euclidean space,
thick at the origin and with connected link. Suppose the tangent link is of pure dimension $d-1$. 
There exists a separating set if and only if the surjective homomorphism of the fc-homology groups 
\begin{center}
$\iota_*:\calH_{d-2} (\thin(X_1),\bo) \to \calH_{d-2}(X_1,\bo)$
\end{center} 
induced by the inclusion mapping $\iota:\thin(X_1) \to X_1$ is not injective. \em
%
%
%
%
%
%
%
%
%
%
%
%
%
%
%
%
%
%
%
%
%
%
%
%
\section{Background - Notations}\label{Section:background}
We present in this section the few definitions and notations we are going to use in the paper.
We also expect here to provide most of the material to present this note  as rather self-contained.

\medskip
In order to avoid any confusion about the notion of cone, we recall the following
\begin{definition}
A subset $C$ of some Euclidean space $\R^p$ is a cone with vertex the origin if there 
exists a subset $L$ of the unit sphere $\Sr^{p-1}$, \em the link of the cone, \em such that
\begin{center}
\vspace{4pt}
$C := \cone(L) = \{tu: u\in L$, and $t>0\}\cup \{\bo\}$.
\end{center}
\end{definition}
\noindent
{\bf Notation:} For any non-zero vector $x$ of some Euclidean space, we denote by $\nu (x)$ 
the oriented direction of the vector $x$
\begin{center}
$\nu (x) := \displaystyle{\frac{x}{|x|}}$.
\end{center}
\begin{definition}
Let $X$ be a subset of some Euclidean space $\R^p$ whose closure $\clos (X)$ contains the origin $\bo$.
The \em tangent cone \em $T_\bo X$ of $X$ at $\bo$ is the closed cone over the subset $\So X$, \em the tangent link
of the subset $X$ at $\bo$ \em defined as
\begin{center}
\vspace{4pt}
$\So X := \{\bu \in \Sr^{p-1}:$ {\rm there exists} $(x_k)_{k\geq 1}\in X\setminus \bo$ {\rm such that}  $\hfill$

\vspace{4pt}
$\hfill x_k \to \bo$ {\rm and}
$\nu (x_k) \to \bu$ {\rm as} $k\to+\infty\}$.
\end{center}
\end{definition}

As already mentioned in the introduction, any subset $X$ of some Euclidean space can be equipped with two metrics: 
The \em outer metric \em measures the distances in the ambient Euclidean space while the \em inner metric \em
measures the distance between two points of $X$ as the infimum of the length of
rectifiable curves within $X$ connecting the two points. These two structure of metric
spaces on $X$ are not necessarily equivalent metric structures. This fact justify the introduction
of the following
\begin{definition}\label{def:n-e}
i) A subset $X$ of $\R^p$ is \em normally embedded \em if the identity mapping between the metric 
spaces $(X,\dou)$ and $(X,\din)$ is bi-Lipschitz. 
\\
ii) Any normally embedded subset $X$ bi-Lipschitz homeomorphic to a subset $X_0$ contained in 
$\R^q$ and equipped with the inner metric is called \em a normally embedded model \em of $X_0$. 
\end{definition}
Throughout this paper we are going to work with subsets definable in an o-minimal structure 
expanding the field of real numbers. For readers not familiar with an o-minimal framework, they can think of 
the category of the semi-algebraic subsets - which form an o-minimal structure - (see \cite{vdDMi,vdD}
for definitions and standard properties). 

\medskip\noindent
{\bf Important Note:} We assume for the rest of the paper 
we are given an o-minimal structure $\calM$ expanding the field of real numbers. Any future occurrence of the 
word \em definable \em will mean definable in this o-minimal structure $\calM$.
\begin{remark}
If $X$ is a compact definable subset of $\R^p$, \em normally embedded \em will mean \em definably
normally embedded \em, that is the bi-Lipschitz homeomorphism in Definition \ref{def:n-e} 
is required to also be a definable mapping.
\end{remark}
A compact definable subset $X_0$ of some Euclidean space $\R^p$ always admits
a \em normally embedded model, \em that is there exists a normally embedded definable
subset $X$ of some $\R^n$ which is definably bi-Lipschitz homeomorphic to the metric space $(X_0,\din)$
(the definable case works exactly as the semi-algebraic case of \cite{BiMo} since the existence 
of a definable pancake decomposition can be proved along the same lines as those of the 
subanalytic case \cite{Kur}, see also \cite{Par,Sh,Co,Paw,Va3}).
We introduce now two notions which will be key for the rest of the paper 
\begin{definition}\label{def:thick-thin}
Let $(X,\bo)$ be the germ at the origin of a closed definable subset of some Euclidean space. Let $d$ 
be the dimension of the set-germ $(X,\bo)$. \\
i) The subset $X$ is \em thick at the origin $\bo$ \em  if the tangent link $\So X$ at $\bo$ has dimension $d-1$.
\\
ii) The subset $X$ is \em thin at the origin $\bo$ \em  if the tangent link $\So X$ at $\bo$ has dimension strictly 
smaller than $d-1$.
\end{definition}
Note that the notion of \em thick \em we propose in Definition \ref{def:thick-thin}
is not as specific and thus looser than that provided in \cite{BNP}.

\medskip
Let $X$ be a definable subset of some Euclidean space $\Rn$. We recall that the germ $(X,x)$ of 
$X$ at any of its point $x$ is topologically conic: 
Indeed, there exists $\ve >0$ such that the intersection $X\cap \bB^n(x,r)$, where $\bB^n(x,r)$ is the 
ambient Euclidean open ball of $\Rn$ 
centered at $x$ of radius $r$, is definably homeomorphic to the cone over the intersection
$X\cap \Srn(x,r)$ for any positive radius $r<\ve$ where $\Srn (x,r)$ is the Euclidean sphere of 
radius $r$ and centered $x$ (\cite{Loj,Mil,vdDMi,vdD,Sh,Co}). 
\begin{definition}
The \em link of \em $X$ at $x$ is the topological space homeomorphic to any intersection $X\cap \Srn(x,\ve)$ 
for any sufficiently small positive $\ve$. By a minor abuse of language we will also call \em link \em 
any intersection  $X\cap \Srn(x,\ve)$ when $\ve$ is small enough.
\end{definition} 
Hereafter, every representative of a germ $(X,x)$ is considered into an ambient Euclidean open ball $\bB(x,r)$ of radius $r<\ve.$
\begin{definition}\label{def:isol-sing}
The germ at the origin $\bo$ of a definable closed subset $X$ of some Euclidean space $\R^p$ is called \em 
a closed definable isolated singularity, \em if $(X\setminus \bo,\bo)$ is the germ of a $C^1$  
submanifold of $\R^p$. 
\end{definition}
Although the notion of isolated singularity is classical, we recall it in the context we are interested 
in to emphasize that in the following sections we will not only deal with a closed isolated singularity but 
also and mainly with definably homeomorphic images, allowing the possibility of singular points outside the origin.  
In particular we will speak very often of a \em normally embedded model \em 
and of a \em simply embedded model \em (see Definition \ref{def:SE-mod}) of the given closed 
definable isolated singularity, which are definably homeomorphic images, with additional properties. 
Since the link of an isolated singularity is always a $C^1$ submanifold, 
the link of any definably homeomorphic image will still be a definable topological
submanifold 

\medskip
We end-up this section recalling the notion of \em spherical blowing-up \em of the origin of $\Rn$ defined 
as the following mapping 
\begin{center}
\vspace{4pt}
$\beta_n : \Srn \times \R_+ \to \Rn\,$, $\,(\bu,r) \to r\bu$.
\vspace{4pt}
\end{center}
The \em strict transform \em of the subset $X$ under the spherical blowing-up $\beta_n$ is the subset of 
$\Srn\times \R_+$ defined as
\begin{center}
\vspace{4pt}
$X':= \clos (\beta_n^{-1}(X\setminus \bo)$.
\vspace{4pt}
\end{center}
The \em boundary $\bdr X'$ \em of the strict transform $X'$ is just the intersection of $X'$ with 
the ambient exceptional divisor $\Srn \times 0$, namely,
\begin{center}
\vspace{4pt}
$\bdr X':=X'\cap (\Srn\times 0) = \So X \times 0$. 
\vspace{4pt}
\end{center}
We would like to insist on the fact that, since the tangent link $\So X$ of $X$ at the origin is the basis 
of the tangent cone $T_\bo X$ of $X$ at the origin, the main interest for us in using a spherical blowing-up 
is to give substance to the tangent link of $X$, embedded in the strict transform $X'$ of $X$ as the boundary 
$\bdr X'$. The strict transform $X'$ allows to see simultaneously the geometric object we are interested in 
($X$) and its tangent cone at the origin ($T_\bo X$).

Note that when there will be no ambiguity on the ambient Euclidean space we are in, we will just write $\beta$.
%
%
%
%
%
%
%
%
%
%
%
%
%
%
%
%
%
%
%
%
\section{Fast contracting cycles and fast contracting homology}\label{Section:FCCaFCH}
We present here the objects and the morphisms of the classification problem 
as well as some of their elementary properties that will allow to
define in Sections \ref{Section:MR} and \ref{Section:TaToTIS} the notion of \em topological
thin zone. \em

We want to classify germs of definable subsets (of some Euclidean space) up to 
a class of definable homeomorphisms with special properties. The simplest model of germs, up to 
these special homeomorphisms, we have in mind are just cones. 
First, we present the objects under our considerations which we want to be preserved by these
special homeomorphisms. Only after, we will define precisely what are these special homeomorphisms.

\medskip
Let $(X,\bo)$ be a germ of definable closed subset at the origin $\bo$ of some Euclidean space and 
with connected link at the origin. 

\smallskip
We recall that any singular homology class of any given bounded definable set can be realized by a definable cycle
\cite{Loj,Har,vdD,Sh,Co}.  
\begin{definition}\label{def:admis-van-hom}
A definable singular $k$-dimensional cycle $\xi$ of $X$ \em supported in the link \em of $X$, that is 
whose support is contained in $X\setminus \bo$, is said to be \em fast-contracting \em if there exists a
definable $(k+1)$-dimensional chain $\eta$ on $X$ bounded by the cycle $\xi$, whose support is thin 
at the origin and contains the origin.
\end{definition}
We will shorten the expression \em fast contracting cycle supported in the link \em into \em fc-cycle. \em

We denote $\calH_k(X,\bo)$ the subset of the $k$-th singular homology group of the link $H_k(X\setminus\bo)$ 
consisting only of the homology classes represented by fc-cycles. We observe that $\calH_k(X,\bo)$
is a subgroup of $H_k(X\setminus\bo)$.
\begin{definition}\label{def:van-hom-grp}
The group $\calH_k(X,\bo)$ is called the \em $k$-th local homology group of the fc-cycles of $X$ at the origin. \em
\end{definition}
Any homology class of the link represented by a fast contracting cycle supported in the 
link will be called a \em fast contracting homology class \em and shortened to \em fc-homology class. \em

\medskip
We do not precise the coefficients group since this will not make any change on what we are going to present and 
demonstrate. What we have in minds are the standard coefficients groups: $\Z/2, \Z, \Q$ and $\R$.

\medskip
Having specified the type of properties we want to classify germs of isolated singularities 
by, we introduce the class of morphisms which goes with our classification problem.

Let $(X,\bo)$ and $(Y,\bo)$ be two germs of closed definable subsets of $\Rn$ and $\R^m$ respectively. 
\begin{definition}\label{def:morphism}
A continuous definable mapping $\Phi:(X,\bo)\to(Y,\bo)$ is called a \em morphism \em
if 

i) the image $\Phi(X\setminus \bo)$ is contained in $Y\setminus\bo$;

ii) the mapping 
\begin{center}
$\beta_m^{-1}\circ \Phi\circ \beta_n: X'\setminus \bdr X' \to Y'\setminus \bdr Y'$
\end{center}
extends as a  continuous definable mapping $\Phi':X'\to Y'$. 
In particular the mapping $\Phi'$ induces 
a continuous definable mapping from the boundary $\bdr X'$ of $X'$ into the boundary $\bdr Y'$ of $Y'$.
\end{definition}
%
An \em epi-morphism \em is a morphism such that the extension mapping is surjective.

A \em mono-morphism \em is a morphism such that the extension mapping is injective.

A morphism which is an epi-morphism and a mono-morphism admits a continuous 
definable inverse mapping which extends as an epi-morphism and a mono-morphism. 
We call such a morphism an \em isomorphism. \em

\smallskip\noindent
{\bf Examples.} 
1) Let $S_\alpha:=\{(x,y,z) \in \R^3: = x^2+y^2 = z^\alpha, z\geq 0\}$ be the $\alpha$-horn, for a real 
number $\alpha >2$. Let $\pi$ be the orthogonal projection from the $\alpha$-horn $S_\alpha$ to the horizontal 
plane $\{z=0\}$. Whence $S_\alpha$ is equipped with the outer metric of $\R^3$, the mapping $\pi$ is definable 
and Lipschitz but is not a morphism. Its inverse $\pi^{-1}$, which is not Lipschitz, is a morphism but is 
not an isomorphism. 

\smallskip\noindent
2) Any definable Lipschitz homeomorphism $(X,\dou) \to (Y,\dou)$ between definable germs 
embedded in some Euclidean spaces 
extends as an isomorphism (see below Proposition \ref{prop:bilip-iso} and \cite{BeLy}). 

\smallskip\noindent 
3) Let $T$ be a closed definable cone with vertex the origin. The homeomorphism $T\ni x \to |x| x \in T$ 
and its inverse are both isomorphisms. Because of its definability and its fast collapsing to the origin, 
this homeomorphism admits a differential-like mapping at the origin $T_\bo X \to T_\bo Y$ as in \cite{BeLy} 
and is mapping the whole tangent cone $T_\bo X$ onto the null vector. Even though the inverse homeomorphism 
does not admit a differential-like mapping, it induces a homeomorphism between the tangent links.

\smallskip\noindent 
4) We use \cite[Example 4.3]{BF3} to emphasize even more the difference between our notion of isomorphism 
and definable outer-bi-Lipschitz homeomorphisms. Consider the family of semi-algebraic isolated 
hypersurface singularities of $\R^{n-1}\times\R_+$ defined as:
\begin{center}
\vspace{4pt}
$H_\lambda : = \{(x,t)\in \R^{n-1}\times \R_+: P(x,t)P(x,-t) = t^\lambda\}$ with $\hfill$

\vspace{4pt}
$\hfill P(x,t) =  -t^2 + (x_1-t)^2 + \sum_{i=2}^{n-1} x_i^2$ and $\lambda\in \Q_{>2}$.  
\vspace{4pt}
\end{center}
This family was used when $n=4$ to exhibit a \em separating set, \em namely the subset $H_\lambda \cap \{x_1 =0\}$,
whose existence obstructs the inner metric conicalness of $H_\lambda$ (see Section 
\ref{Section:examples} for the relation of separating set and fast contracting homology).
Note that if $\lambda'>\lambda>2$ the subsets $H_{\lambda'}$ and $H_\lambda$ cannot be outer bi-Lipschitz 
homeomorphic. Indeed, let $\Phi_\alpha:\R^{n-1}\times \R_+\to\R^{n-1}\times \R_+$ be the definable mapping defined as $(x,t)\to
(t^\alpha x,t^{\alpha+1})$ for $\alpha$ a positive rational number. similarly to example 3), this mapping induces 
an isomorphism from $H_\lambda$ onto $H_{\lambda'}$ with $\lambda' := \lambda + \alpha(\lambda -2)$. 
 
\smallskip
As already suggested by the third and fourth example above, the notion of isomorphisms we introduce here is much more 
general than the notion of definable outer bi-Lipschitz homeomorphisms. 

\smallskip
As expected, one of the interests of the fc-homology groups lies in the following result:
\begin{proposition}\label{prop:homom-fc-homol}
Let $(X,\bo)$ and $(Y,\bo)$ be germs of closed definable subsets of some Euclidean spaces.
For each $k\leq \dim Y$, a morphism $\Phi:(X,\bo)\to(Y,\bo)$ induces a homomorphism 
$\Phi_*: \calH_k(X,\bo)\to\calH_k(Y,\bo)$.
If $\Phi$ is an isomorphism then  $\Phi_*$ is also an isomorphism of groups.
\end{proposition}
\begin{proof}
We recall that the dimension of the image of a definable subset by a definable and continuous mapping is 
lower than or equal to the dimension of the subset. The image of a definable $k$-chain $\xi$ by the mapping 
$\Phi$ is a definable $k$-chain. Since the cycle $\xi$ bounds a thin chain $\eta$ whose support contains the origin,
its image $\Phi(\xi)$ bounds the image $\Phi (\eta)$. When embedded in the boundary of the strict transform $X'$,
the image of the tangent link of the support $\supp(\eta)$ by the mapping $\Phi'$ is of dimension lower than or equal 
to the dimension of the tangent link of $\supp(\eta)$, thus the image 
$\Phi(\xi)$ is contracted by a thin chain, namely the image $\Phi (\eta)$.
\end{proof}
We can even be more specific requiring an additional information on the collapsing 
phenomenon of a given fc-homology class taking into account a measure of the collapsing. 
\begin{definition}\label{def:dim-drop}
1) The \em dimension drop at the origin \em of a fc-cycle $\xi$ is the maximum of $\dim \supp (\eta) - 
\dim T_0 \supp (\eta)$ taken over the definable chains $\eta$ contracting $\xi$ in $X$ whose support 
contains the origin and is thin. 
\\
2)
The \em dimension drop of a fc-homology class $[c]$ \em is the maximum of the dimension drop at the origin of the fc-cycles
homologous to $[c]$. 
\end{definition}
For any degree $k = 1,\ldots,\dim X$, and any dimension drop $\delta =1,\ldots,k$, the 
fc-homology classes of $\calH_k(X,\bo)$  whose dimension drop is larger than or equal to $\delta$
form a subgroup that we denote $\calH_k^\delta(X,\bo)$. We find the following inclusions 
\begin{center}
\vspace{4pt}
$\calH_k^k (X,\bo) \subset \ldots \subset \calH_k^1 (X, \bo) =\calH_k (X, \bo)$.
\vspace{4pt}
\end{center}
Using Proposition \ref{prop:homom-fc-homol} and its proof, we refine its statement in the 
following way
\begin{proposition}\label{prop:homom-fc-homol-degree}
Let $(X,\bo)$ and $(Y,\bo)$ be germs of closed definable subsets of some Euclidean spaces. 
For each degree $k\leq \dim (Y,\bo)$ and each 
dimension drop $\delta =1,\ldots,k$, any morphism $\Phi:(X,\bo)\to(Y,\bo)$ induces a homomorphism 
$\Phi_*: \calH_k^\delta(X,\bo)\to\calH_k^\delta(Y,\bo)$.
\end{proposition}
Finally let us remark that when the tangent link of the subset $X$ at the origin has dimension 
$e-1 < \dim (X,\bo) -1$, then for any degree $k \geq e$ we necessarily find  
$\calH_k(X,\bo) = H_k(X\setminus \bo)$. In particular if $e= 1$, 
we find $\calH_k^k (X,\bo) = \calH_k(X,\bo) = H_k(X\setminus \bo)$ for all $k$. 

We check easily that any non-trivial fc-homology class can be realized 
as Metric homology class \cite{BiBr1,BiBr2} and as Vanishing homology class \cite{Va3}.
%
%
%
%
%
%
%
%
%
%
%
%
%
%
%
%
%
%
%
%
%
%
%
%
%
%
%
%
%
%
%
%
\section{Simple tangent directions}\label{Section:STDaBT}
The aim of this section is multiple. First we characterize the isomorphism equivalence class 
of a cone over a topological definable submanifold. This will allow to show in Section \ref{Section:MR}
that cones do not have fast contracting homology (Theorem \ref{thm:FF-polyhedra} and Corollary \ref{cor:cone-no-hom}).
Second, this description result will allow to define the non-simple tangent direction locus, 
which is precisely the locus of the tangent link over which the fc-cycles will collapse and in a neighborhood of which 
the thin zone we are looking for is structured. 

\medskip
Let $(X,\bo)$ be a closed definable subset of some Euclidean space with connected link at the origin $\bo$.

\smallskip\noindent
{\bf Notation:} The set of points of the tangent link $\So X$ at which the germ of $\So X$ is the germ 
of topological submanifold of maximal dimension is denoted by $\So^* X$. Note that
$\So^* X$ is open in the tangent link, definable and of maximal dimension. 

\medskip
We begin with the next result in order to have constantly present in mind 
that the class of definable homeomorphisms extending as isomorphisms we are using 
contains the definable outer-bi-Lipschitz homeomorphism.
\begin{proposition}\label{prop:bilip-iso}
Let $(X,\bo)$ be the germ of a closed definable subset with connected link at the origin $\bo$. 
Let $\Phi: (X,\bo) \to (Y,\bo)$ be a definable  bi-Lipschitz homeomorphism for the outer metrics.
The homeomorphism $\Phi$ is an isomorphism.
\end{proposition}
\begin{proof}
In their beautiful paper, Bernig \& Lytchak, by taking the derivative in $t=0$ of 
a definable curve $\gamma$ of $(X,\bo)$ and its image $\Phi\circ \gamma$ of $(Y,\bo)$ both starting at the origin,
define a definable homogeneous mapping $D_\bo \Phi :T_\bo X \to T_\bo Y$, which also happens 
to be a bi-Lipschitz homogeneous homeomorphism \cite{BeLy}. 
Namely, given a continuous definable curve $\gamma:[0,\ve[ \to X$ starting at $\bo$ with tangent 
vector $\bu$ at $t=0$, the image vector $D_0\Phi \cdot \bu$ is defined simply as the limit 
$\lim_{t\to0} \frac{\ud}{\ud t}  \Phi\circ \gamma (t)$. 
The homogeneity and the continuity of the mapping $D_\bo \Phi$ induce 
a definable homeomorphism $\So X \to \So Y$ defined by the normalization of $D_\bo\Phi$
\begin{center}
\vspace{4pt}
$\nu_\bo\Phi:\bu \to \displaystyle{\frac{D_\bo\Phi\cdot\bu}{|D_\bo\Phi\cdot\bu|}}$
\vspace{4pt}
\end{center}
Thus when a sequence $(x_k)_k$ of points of $X$ goes to the origin such that the associated sequence of oriented directions 
$(\nu(x_k))_k$ goes to $\bu$, we can write 
\begin{center}
\vspace{4pt}
$\Phi (x_k) = |x_k| D_\bo \Phi\cdot \bu + o(|x_k|)$.
\vspace{4pt}
\end{center}
Assuming that the subsets $X$ and $Y$ are embedded respectively in $\Rn$ and in $\R^p$, we see that the definable 
mapping $\beta_p^{-1}\circ\Phi\circ\beta_n$ when restricted to open strict transform $X'\setminus\bdr X'$ is defined 
as 
\vspace{4pt}
\begin{equation}\label{eq:lift-phi}
\vspace{4pt}
X'\setminus \bdr X' \ni (r,\bu) \to \left( \frac{\Phi(r\bu)}{|\Phi(r\bu)|},|\Phi(r\bu)|\right) \in Y'\setminus \bdr Y'.
\end{equation}
We just have to check that this mapping extends continuously as the homeomorphism
$\nu_\bo\Phi\times0:\bdr X' \to \bdr Y'$ defined as $(\bu,0) \to (\nu_\bo \Phi (\bu),0)$ which is immediate from 
the definition of the mappings $D_\bo \Phi$ and $\nu_\bo \Phi$.
\end{proof}
\begin{remark}
1) If $\Srn\times \R_+$ is equipped with the outer metric from the Euclidean metric on $\R^n\times\R_+$,
then the definable bi-Lipschitz homeomorphism of Proposition \ref{prop:bilip-iso} 
extends as a definable bi-Lipschitz homeomorphism $(X',\bdr X') \to (Y',\bdr Y')$ for the outer metrics. 
\\
2) Proposition \ref{prop:bilip-iso} is very specific to isomorphisms induced by 
an outer definable bi-Lipschitz homeomorphism between closed definable germs. Isomorphisms 
usually do not even give a well defined mapping between tangent cones as seen in Example 3) and 4) of Section 
\ref{Section:FCCaFCH}, even though by definition they induce a homeomorphism between the tangent links.
\\
3) Any definable bi-Lipschitz homeomorphism $\Phi: (X,\din)\to (X_0,\din)$, where $X$ is a normally embedded 
model of closed definable isolated singularity $X_0$ with connected link, 
induces a surjective definable homogeneous finite-to-one mapping $D_\bo \Phi:T_\bo X \to T_\bo X_0$ (see \cite{BeLy}).
Thus the mapping $\Phi$ extends as a morphism.
\end{remark}
If a normally embedded model $X$ of a closed definable isolated singularity $X_0$ with 
connected link at the origin $\bo$ is isomorphic to another definable germ  $X_1$, the
isomorphism induces a definable homeomorphism $X_1 \to X_0$ which extends as a morphism 
and thus induces a surjective continuous definable finite-to-one mapping $\So X_1 \to \So X_0$.  
These facts and Proposition \ref{prop:bilip-iso} justify the introduction of the following - now 
well defined - notion
\begin{definition}\label{def:SE-mod}
Let $X$ be a normally embedded model of a closed definable isolated singularity $X_0$ 
and with connected link at $\bo$.

i) A subset $X_1$ of some Euclidean space is called \em a simply embedded model of $X_0$ at the origin
 \em if it is isomorphic to any normally embedded model of the singularity $X_0$.

ii) A definable homeomorphism $X_1\to X_0$ is a \em simple embedding morphism \em if it is a morphism 
obtained as the composition of a bi-Lipschitz mapping $(X,\din) \to (X_0,\din)$ from a normally embedded 
model $X$ of $X_0$ and of an isomorphism $X_1 \to X$.  
\end{definition}
A very pleasant property of simply embedded models of an isolated singularity is the following
\begin{proposition}\label{prop:SE-iso-NE}
Any normally embedded model of a simply embedded model of a definable closed isolated singularity 
with connected link at the origin is isomorphic to the simply embedded model itself.
\end{proposition}
\begin{proof}
Let $X_0$ be a closed definable isolated singularity with connected link at the origin.
Let $X_1$ be a simply embedded model of $X_0$ and $X$ be a normally embedded model of $X_0$ isomorphic 
to $X_1$. Let $Y$ be a normally embedded model of $X_1$ and let $\Phi:(Y,\din)\to (X_1,\din)$ be a 
definable bi-Lipschitz homeomorphism. It induces a surjective continuous definable finite-to-one mapping
$\phi:\So Y \to \So X_1$. Let $h:X_1\to X$ be the isomorphism and let write $h'$ for the homeomorphism
$X_1' \to X'$. We recall that the composed mapping $\Psi:= h\circ\Phi$ is a definable homeomorphism and 
extends as a morphism $\Psi'$. In particular the mapping $\psi'$ induced by the extension $\Phi'$ between the 
boundaries of the strict transforms of $Y$ and $X$ respectively, namely
$\psi' := h'\circ(\phi\times 0) :\bdr Y' \to \bdr X'$, is finite-to-one and surjective.
Let $x'$ be a boundary point of $X'$ and let $y_1',y_2'$ be two pre-images of $x'$ by $\psi'$.
Let $\calV$ be a small neighborhood of $x'$ in $X'$ and let $\calU$ be its pre-image by $\Psi'$.
Since the subset $Y'$ is contained in some product $\Sr^{p-1}\times \R_+$ we equip it with the outer metric     
of this product. Let $C_1$ and $C_2$ be two continuous definable curves lying in $\calU$ with 
respective starting point $y_1'$ and $y_2'$. We can assume these curve are parametrized 
as $t\to c_i (t)$ by the height function (projection onto the component $\R_+$ of the product).
Let $M(C_1,C_2)$ be the set of points $z'$ of the neighborhood $\calU$ such that $z'$ is equidistant to the points
of $C_1$ and $C_2$ with the same height as $z'$. This subset is definable, not empty and thus contains 
a continuous definable curve $C$ such that its starting point lies in the boundary $\calU \cap \bdr Y'$ and
is equidistant to $y_1'$ and $y_2'$. The image $\Psi'(C)$ of $C$ is contained in the open subset $\calV$ 
and starts at a point of the boundary $\calV \cap \bdr X'$. 
Let $(\delta_k)_k$ be a sequence of positive real numbers decreasing to $0$.
Let $\calV_k$ be the open subset of points of $X'$ with distance to $z'$ is strictly less than $\delta_k$.
Let $\calU_k$ be the pre-image $(\Psi')^{-1}(\calV_k)$. 
We construct a sequence of definable curves $(C_k)_{k >0}$ of the strict transform $Y'$ 
with starting point $y_k$ in the boundary $\bdr Y'\cap\calU_k$, equidistant from $y_1'$ and $y_2'$.  
We can assume that the sequence $(y_k)_k$ converges to $y'$. Since the continuous definable curve $\Psi'(C_k)$ 
starts at a point $z_k$ of the boundary $\bdr X'\cap \calU_k$, the sequence $(z_k)_k$ converges to $z'$ as 
$k$ goes to $\infty$. Thus $y'$ is a pre-image of $z'$ equidistant from $y_1'$ and $y_2'$.
Unless the boundary mapping $\psi'$ is injective, we can construct as finitely many pre-images 
of $x'$ as we want. Thus $\psi'$ is a homeomorphism, and consequently $X_1$ is isomorphic to any 
of its normally embedded model.
\end{proof}

Any normally embedded model of a closed definable isolated singularity is 
simply embedded. The converse is not true as seen in the following example in $\R^3$ 
\begin{center}
\vspace{4pt}
$X = \clos (\{x^2+y^2 = z^2, z>0\})$ and 
 $X_1 := \clos(\{y^2z^2 + x^4 = zx^3, z> 0\})$.
\vspace{4pt}
\end{center}
These cones are clearly isomorphic but $X_1$ is not normally embedded, since the compact plane curve
$\{y^2+x^4=x^3\}$, basis of the cone $X_1$, is not normally embedded.

\medskip
The following result although obvious makes also sense of the notion of isomorphism 
we have introduced in Section \ref{Section:FCCaFCH}.
\begin{proposition}\label{prop:SE-universal}
1) Any two simply embedded models of a same closed definable isolated singularity 
and with connected link are isomorphic. 
\\
2) If for $i=1,2$, the mappings $\Psi_i:X_i \to X_0$ are simple embedding morphisms of a closed definable isolated singularity 
$X_0$ of some Euclidean space and with connected link, then the mapping $\Psi_2^{-1}\circ \Psi_1:X_1\to X_2$ 
is an isomorphism.  
\end{proposition}
\begin{proof}
Each model is isomorphic to any normally embedded model of the singularity $X_0$. Since the composition of two 
isomorphisms is again an isomorphism, we get the result. The second point just comes from the definition of a simple
embedding morphism.
\end{proof}
\noindent
{\bf Important note:} from now on we will only deal with definable closed isolated singularities {\bf thick} at 
the origin and with connected link. 
\begin{proposition}\label{prop:simple-dense}
Let $X$ be a normally embedded model of a closed definable isolated singularity thick at the origin 
and with connected link. 
There exists a smallest closed definable subset $\Sigma$ of the tangent link $\So X$ of 
codimension at least one such that for each connected component $Z$ of the complement $\So X \setminus\Sigma$, 
the topological type of the definable family of germs $((X',z))_{\{z\in Z\times0\}}$ is constant along the component $Z$
and the germ at any point $(z,0)$ of $Z\times 0$ of the pair $(X',\bdr X')$ is the germ of a topological submanifold with boundary.
\end{proposition}
\begin{proof}
The classical theorems of topological equisingularity ensures that the definable family of
germs $((X',x'))_{\{x'\in\bdr X'\}}$ has finitely many topological types, see \cite{Har,vdD,Sh,Co}.
In particular the compact parameter space $\bdr X'$ is partitioned into finitely many definable pieces
along which the topological type of the germs $(X',x')$ is constant (see \cite{Har,vdD,Sh,Co}).
This partitions the tangent link $\So X$ into finitely many connected definable topological submanifolds 
$(Z_i)_{i\in I}$ such that along each piece $Z_i$ the topological type of the family of germs $(X',x')_{\{x'\in Z_i\times 0\}}$
is constant. We call the elements of this partition strata. Let $Z$ be a stratum of maximal dimension,
then it is open in then tangent link $\So X$. Since a definable Whitney stratification of the pair 
$(X',\bdr X')$ exists \cite{vdDMi,vdD} and since the subset $X$ is normally embedded, the topological 
type of the strict transform $X'$ at any point of the stratum $Z\times 0$ is that of a topological submanifold with boundary.
We take $\Sigma$ to be the complement of the union of the strata of maximal dimension along which the 
topological type is constant.
\end{proof}
Since a simply embedded singularity is isomorphic to a normally embedded isolated singularity 
the strict transforms under the respective spherical blowings-up are homeomorphic, thus the 
topological nature of the germ of the strict transform at a point of the boundary is transported
by the extension homeomorphism. We can now introduce the following 
\begin{definition}\label{def:simple-dir}
Let $X_1$ be a simply embedded model of a definable closed isolated singularity $X_0$ 
thick at the origin and with connected link. 

i) A point of the boundary $\bdr X_1'$ of the strict transform $X_1'$ 
at which the germ of $X_1'$ is a topological submanifold with boundary is called 
a \em simple point. \em We denote by $\smp(\bdr X_1')$ the set of simple points of $\bdr X'$ 
and by $\ns (\bdr X_1')$ its complement in $\bdr X_1'$.

ii) A tangent direction corresponding to a simple point of the boundary $\bdr X_1'$ is called 
a \em simple tangent direction. \em
Let $\Sos X_1$ be the subset of \em simple tangent direction of $X_1$ at the origin. \em
The complement set of the simple tangent directions at the origin in the tangent 
link, is the \em locus of non-simple tangent directions \em and is denoted $\Sns X_1$.
\end{definition}
As already explained isomorphic simply embedded models will have definably homeomorphic simple tangent direction loci.
 
Since the boundary $\bdr X_1' = \So X_1 \times 0$ is an embedding of the tangent link $\So X_1$, we will only work 
with the tangent link. Note that we obviously find that $\smp (\bdr X_1') = \Sos X_1 \times 0$ and $\ns(\bdr X_1') = 
\Sns X_1 \times 0$.

\bigskip
We will see below and in the next two sections that the locus of non-simple tangent directions 
does carry a lot of information about the fast contracting homology classes.  
As our first step towards this goal we recall Brown Theorem about the existence of collar 
neighborhoods in metric spaces, since it will be very useful to define the notion of thick zone.

\smallskip
A subset $B$ of a topological space $Y$ is \em collared \em  if there exists a
homeomorphism $h:B\times [0,1[ \to \calU$ onto a neighborhood $\calU$ of $B$ such that
$h(b,0) = b$ for each point $b$ of $B$. A subset $B$ of the topological space $Y$ is \em locally collared \em
if it can be covered by a collection of open subsets $B_i$ of $B$ such that
each $B_i$ is collared. The main result of Brown \cite{Br} is
\begin{theorem}[\cite{Br}]\label{thm:Brown}
A locally collared subset of a metric space is collared.
\end{theorem}
Reading the simple and elegant proof of this theorem, we observe that this result 
can be achieved definably if we start with definable data. 

Brown Theorem provides a pleasant property of simply embedded models
of a given thick isolated singularity whose tangent link consists only of 
simple tangent directions:   
\begin{lemma}\label{lem:loc-conic-homeo-cone}
Let $X_1$ be a simply embedded model of a closed definable isolated singularity of some Euclidean 
space, thick at the origin and with connected link $\bo$. If the tangent link $\So X_1$ is a topological 
submanifold, then it consists only of simple tangent directions if and only if the subset $X_1$ is 
isomorphic to its tangent cone $T_\bo X_1$.
\end{lemma}
This simple result is the basis of our main results of Sections \ref{Section:MR}
and \ref{Section:TaToTIS}.
\begin{proof}[Proof of Lemma \ref{lem:loc-conic-homeo-cone}.]
If $X_1$ isomorphic to its tangent cone, then its tangent link consists only of
simple points.

\smallskip
We equip the strict transform $X_1'$ and its boundary $\bdr X_1'$ with the outer metric of the
metric space $\Srn \times \R_+$ coming from its inclusion in $\Rn\times\R_+$.

Brown Theorem implies the existence of a positive real number $\ve$ and of a definable
homeomorphism $\Phi:\So X_1 \times [0,\ve[ \to \calV$ where $\calV$ is an open neighborhood of
$\bdr X_1'$ in $X_1'$ and is also definable. Moreover we also know that $\Phi(\bu,0) = (\bu,0)$.
Composing with the blowing-down mappings at the target (of $\Phi$) and at the source (of $\Phi$) 
we obtain a definable homeomorphism
\begin{center}
$\beta\circ\Phi\circ\beta^{-1}:T_\bo X_1 \setminus \bo \to X_1 \setminus\bo$,
\end{center}
which extends continuously and definably at $\bo$.
\end{proof}
\noindent
{\bf Notation:} The \em $\ve$-cone \em of a subset $L$ of the unit sphere $\Srn$ is the subset of $\Rn$ defined as 
\begin{center}
\vspace{4pt}
$\cone(L,\ve) := \clos(\{x \in \Rn\setminus\bo: \dist (\nu (x),L) < \ve\})$.
\vspace{4pt}
\end{center} 
With the same principles as those of the second part of the proof of Lemma \ref{lem:loc-conic-homeo-cone}, 
we deduce the following
\begin{corollary}\label{cor:isom-tgt-cone}
Let $X_1$ be a simply embedded model of a closed definable isolated singularity of some Euclidean space, 
thick at the origin and with connected link $\bo$. For a positive real number $\ve$ small enough, each 
connected component of the complement $X_1 \setminus \cone(\Sns X_1,\ve)$ is isomorphic to its tangent cone. 
\end{corollary}
%
%
%
%
%
%
%
%
%
%
%
%
%
%
%
%
%
%
%
%
%
%
%
%
%
%
%
%
%
%
\section{Collapsing homology locus of thick isolated 
singularities}\label{Section:MR}
We present here the first version of the main result of this paper: Theorem \ref{thm:main} states that
the fast contracting homology of a simply embedded model $X_1$ of a closed definable thick isolated singularity 
$X_0$ with connected link can be realized in a thin zone of $X_1$ organized around the non-simple tangent 
directions. Prior to getting there, we prove Theorem \ref{thm:FF-polyhedra} which guarantees that cones do 
not have fast contracting homology, fact absolutely essential to show Theorem \ref{thm:main}

\medskip
The next result is a polyhedral version of 
Federer-Flemming Theorem on the triviality of cycles of
small volume in a smooth compact Riemannian manifold \cite{FF}.
By $\vol_k$ is meant the $k$-dimensional volume taken with the outer metric.

\begin{theorem}\label{thm:FF-polyhedra} Let $\Gamma$ be a compact polyhedron defined as the union of the simplices
of a finite set $T$.
Let $\xi$ be a definable  $k$-cycle in $\Gamma$ such that for any positive $\ve$ small enough there exists a
homologous definable $k$-cycle $\xi_{\ve}$ in $\Gamma$ whose $k$-dimensional volume
$\vol_k(\supp(\xi_{\ve}))$ is smaller than $\ve$. Then, $\xi$ bounds a definable $(k+1)$-chain of $\Gamma$.
\end{theorem}

\begin{proof} The idea of the proof is the similar to that of the main result of \cite{BiBr2}.
We are going to show that the cycle $\xi$ is homologous to a definable $k$-cycle
living in the $(k-1)$-skeleton of the polyhedron $\Gamma$. The dimension of $\Gamma$ is $d$.

\smallskip
First, observe there exists a positive real number $L$
such that: For any simplex $\sigma$ of $T$ and any interior point $x_0$ of $\sigma$, the mapping
$$\Phi_{\sigma,x_0}:\sigma\setminus\{x_0\}\rightarrow\partial\sigma$$
defined as the intersection point of the half-line $\R_+ (x-x_0)$ with the boundary $\partial\sigma$,
is a $L$-Lipschitz mapping. The existence of such uniform constant $L$ is due to the fact that
the family $(\Phi_{\sigma,x})_{x\in\sigma,\sigma\in T}$ is a semi-algebraic family of piecewise
linear mappings over the same compact semi-algebraic source.

Let $\ve$ be positive and strictly smaller than
$$\frac{1}{L^{(d-k)k}}\inf\{\vol_{k}(\sigma) \ : \ \sigma \mbox{ is a $k$-simplex of } T\}.$$
Let $\xi_{\ve}$ be a definable $k$-cycle in $\Gamma$ homologous to $\xi$ and whose
volume $\vol_k(\supp(\xi_{\ve}))$ is smaller than $\ve$. Let $\sigma$ be a $d$-simplex of $T$.
Since the interior of the simplex $\sigma$ is not contained in the support of the cycle $\xi_{\ve}$,
there exists a point $x_0$ in the interior of $\sigma$ such that the support of $\xi_\ve$ is contained in
$\sigma \setminus\{x_0\}$. So, we find
$$\vol_k(\Phi_{\sigma,x_0}(\xi_\ve\cap\sigma))\leq L^k \vol_k(\supp(\xi_\ve)\cap\sigma).$$
Hence, the $k$-cycle $\xi_{\ve}$ is homologous to a definable $k$-cycle $\xi_{d-1}$ which lives on 
the $d-1$-skeleton of the polyhedron $\Gamma$ and whose volume satisfies
$$\vol_k(\supp(\xi_{d-1}))\leq L^k \vol_k(\supp(\xi_{\ve})).$$
If $d-1=k-1$ the proof ends here. If not, since the volume of the support of the cycle $\xi_{d-1}$ is small, 
we are dealing with a $(d-1)$-simplex $\sigma$ of $T$, whose interior is not contained in the support of $\xi_{d-1}$,
hence there exists a point $x_1$ in the interior of $\sigma$ such that the support of $\xi_{d-1}$ is contained in
$\sigma\setminus\{x_1\}$. We deduce the following inequality
$$\vol_k(\Phi_{\sigma,x_1}(\xi_{d-1}\cap\sigma))\leq L^k \vol_k(\supp(\xi_{d-1})\cap\sigma).$$
Hence, the cycle $\xi_{d-1}$ is homologous to a definable $k$-cycle $\xi_{d-2}$ which lies in
the ($d-2$)-skeleton of $\Gamma$ and satisfies
$$\vol_k (\supp(\xi_{d-2}))\leq L^k \vol_k(\supp(\xi_{d-1}))\leq L^{2k} \vol_k(\supp(\xi_{\ve})).$$
Iterating this process, we conclude that the cycle $\xi$ is homologous to a definable $k$-cycle $\xi_k$ living
in the $k$-skeleton of the polyhedron $\Gamma$ and satisfies the following inequality
$$\vol_k(\supp(\xi_k))\leq L^{(d-k)k} \vol_k(\supp(\xi_{\ve})).$$
Since
$$L^{(d-k)k} \vol_k(\supp(\xi_{\ve}))<\inf\{\vol_k(\sigma) \ : \ \sigma\in T \ \mbox{is a k-simplex}\},$$
given a $k$-simplex $\sigma$ of  $T$, its interior is not contained in the support of the cycle $\xi_k$, thus there
exists $x_k$ a point in the interior of $\sigma$ such that the support of $\xi_k$ is contained in
$\sigma\setminus\{x_k\}$. Considering the projection mapping
$$\Phi_{\sigma,x_k}:\sigma\setminus\{x_k\}\rightarrow\partial\sigma$$
where $\sigma$ is a $k$-simplex of $T$ and $x_k$ is a point in the interior of $\sigma$ such that the support of
the cycle $\xi_k$ is contained in $\sigma\setminus\{x_k\}$.
We thus obtain a definable $k$-cycle $\xi_{k-1}$ homologous to $\xi$ and contained in the ($k-1$)-skeleton of 
the polyhedron $\Gamma$. Necessarily $\xi$ was trivial.
\end{proof}
We deduce the following two corollaries.
\begin{corollary} \label{cor:FF-defin} Let $\Gamma$ be a compact definable set of some Euclidean space. 
Let $\xi$ be a definable $k$-dimensional cycle in $\Gamma$ such that for any positive $\ve$ there exists 
a definable $k$-dimensional cycle $\xi_{\ve}$ in the set $\Gamma$ homologous to $\xi$ and whose volume 
$\vol_{k}(\supp(\xi_{\ve}))$ is smaller than $\ve$. Then the cycle $\xi$ bounds a definable 
$(k+1)$-chain in $\Gamma$.
\end{corollary}
Let $Z$ be a closed definable subset of $\Rn$ containing the origin $\bo$ and of dimension $d$
at the origin. Let $Z(\ve)$ be the intersection $Z\cap \bB^n (\bo,\ve)$. 
We recall that the normalized $d$-volume $\ve^{-d} \vol_d (Z(\ve))$ tends to $0$ as the radius $\ve$ goes to $0$ 
if and only if the tangent cone $T_\bo Z$ has dimension strictly less than $d$.  
\begin{corollary}\label{cor:cone-no-hom}
A germ of a closed definable subset of some Euclidean space isomorphic 
to a cone has trivial fast contracting homology.
\end{corollary}
\begin{proof}[Proof of Corollary \ref{cor:cone-no-hom}.]
We will use Theorem \ref{thm:FF-polyhedra}. We can assume $Z$ is contained in $\Rn$ and its tangent link $\So Z$
is connected. We work within the intersection of $Z$ with a closed ambient Euclidean ball $\clos(\bB^n(\bo,\ve))$. Let 
$\xi$ be a fast contracting $k$-cycle and let $\eta$ be the chain it bounds in $Z$. For a positive small real number $t$ 
we define $\xi_t$ as the $k$-cycle whose support is the intersection $\supp(\eta) \cap \Sr_t^{n-1}$,
where $\Sr_t^{n-1}$ is the $(n-1)$-dimensional Euclidean sphere centered at $\bo$ of radius $t$,
and the coefficients are those from $\eta$ when the intersection is not empty.
Thus $\xi_t$ is a cycle homologous to $\xi$. Since the germ $Z$ is the homeomorphic 
image of a cone by an isomorphism, the support $\supp(\xi_t)$ is contained in the scaled tangent link 
$t(\So Z)\subset \Sr_t^{n-1}$. Since the support of $\eta$ is thin at the origin, the 
Hausdorff limit at $t=0$ of the rescaled support $\frac{1}{t} \supp (\xi_t)$ tends to a definable subset 
of dimension at most $k-1$ contained in the tangent link $\So Z$. We apply Corollary \ref{cor:FF-defin} to the tangent 
link $\So Z$ and the family of cycles $(\frac{1}{t} \xi_t)_{\{0<t\ll 1\}}$ and get the result.
\end{proof}
Note that Corollary \ref{cor:cone-no-hom} above gives a very quick proof to the main result 
\cite[Theorem 2.6]{BFGoS} of the paper \cite{BFGoS}.

\smallskip
We can now state the main result of the paper. We recall that $\Sns X_1$ is the locus of non-simple 
tangent directions (see Definition \ref{def:simple-dir}).
\begin{theorem}\label{thm:main}
Let $X_1$ be a simply embedded model of a closed definable isolated singularity of some Euclidean
space, thick at the origin and with connected link. 
For any positive real number $\ve$ small enough, the inclusion mapping induces the following exact sequence 
\begin{equation}\label{eq:main-thm}
\calH_\bullet (X_1\cap \cone(\Sns X_1, \ve),\bo) \to \calH_\bullet (X_1,\bo) \to 0.
\vspace{4pt}
\end{equation}
\end{theorem}
Before getting into the proof let us mention the straightforward following:
\begin{corollary}\label{cor:main-cor}
For any degree $k=1,\ldots,\dim X_1$ and for any dimension drop $\delta = 1,\ldots,k$,
for any positive real number $\ve$ small enough, the inclusion mapping induces the following exact sequence 
\begin{equation}\label{eq:main-1}
\calH_k^\delta (X_1\cap \cone(\Sns X_1, \ve),\bo) \to \calH_k^\delta (X_1,\bo) \to 0.
\vspace{4pt}
\end{equation}
\end{corollary}
Theorem \ref{thm:main} is a consequence of the next result.
\begin{lemma}\label{lem:homology-in-thin-non-simply}
Let $\xi$ be a $k$-dimensional definable fc-cycle of $X_1$.
Then for any $\ve$ small enough, the homology class of the cycle $\xi$ can be represented by a fast
contracting definable cycle supported in the link of $\cone(\Sns X_1,\ve)$. 
\end{lemma}
\begin{proof}
For $\ve$ positive and small enough the topological type of the germ 
at the origin of the complement $X_1 \setminus \cone(\Sns X_1, \ve)$ is constant. It consists in finitely 
many closed connected components each of which is isomorphic to its tangent cone by
Corollary \ref{cor:isom-tgt-cone}.

\smallskip
Let $\calU_\ve := \cone (\Sns X_1,\ve)$ and $\bdr \calU_\ve := \clos (\calU_\ve)\setminus\calU_\ve$.
\smallskip
Let $\xi$ be a definable fc-cycle of dimension $k$ bounding a definable  $(k+1)$-chain $\eta$
whose support is thin and contains the origin. 
There exists a definable triangulation of the subset $X_1$ such that $\eta = \sum_{\sigma\in K} a_\sigma \sigma$ and
where each of the following intersections $\supp(\eta)\cap(X_1\setminus\calU_\ve)$, $\supp(\eta)\cap\bdr\calU_\ve$ and
$\supp(\eta)\cap\clos(\calU_\ve)$ is a finite union of simplices.

We observe that the support of each simplex $\sigma$ of $K$ of dimension $k+1$ is thin at the origin.

Let $K^\bdr$ be the set of simplices of the chain $\eta$ contained in $\bdr\calU_\ve$,
$K^+$ be the set of simplices of $\eta$ contained in $X_1\setminus \calU_\ve$ but not in $\bdr\calU_\ve$, and
$K^-$ be the set of of simplices of $\eta$ contained in $\clos(\calU_\ve)$ but not in $\bdr\calU_\ve$.
In particular the chain $\eta$ is presented as the following sum
\begin{center}
$\eta = \eta^+ + \eta^- + \eta^\partial,$
\end{center}
where $\eta^* = \sum_{\sigma\in K^*} a_\sigma \sigma$ for $*\in \{-,+,\bdr\}$. We then check that the cycle $\xi$
decomposes as 
\begin{center}
$\xi = \partial \eta = \xi^+ + \xi^\partial + \xi^-$
\end{center}
with $\xi^* =\bdr \eta^*$ where $*\in \{-,+,\bdr\}$.
In particular $\xi^+$, $\xi^\partial$ and $\xi^-$ are fc-cycles at $\bo$.
By Corollary \ref{cor:isom-tgt-cone} and Corollary \ref{cor:cone-no-hom} 
we deduce that the cycles $\xi^+$ and $\xi^\partial$ are trivial and thus $\xi$ is homologous to $\xi^-$ which is 
contained in $\calU_{\ve'} = \cone (\Sns X,\ve')$ for any $\ve'>\ve$.
\end{proof}
%
%
%
%
%
%
%
%
%
%
%
%
%
%
%
%
%
%
%
%
%
%
%
%
%
%
%
%
%
%
%
%
\section{Topological thick and thin decomposition theorem for thick isolated 
singularities}\label{Section:TaToTIS}
As a consequence of the results of Section \ref{Section:MR}, we will produce a systematic 
decomposition of the isolated singularity we are working with into a \em topological thick part 
\em and a \em topological thin part. \em As already proved in Section \ref{Section:MR}, any fc-homology 
class can be realized in the thin zone.

As in \cite{BNP}, this decomposition is not unique stricto sensu.
It is a one-parameter definable family of open subsets such that the Hausdorff limit at
$0$ of the family of tangent cones at the origin is the cone over the locus of the non-simple 
tangent directions, thus is thin (see Definition \ref{def:thin-zone}). 
We will see that it is consistent under automorphisms and its topology does not depend
on the choice of any element of the family. 
\\
In order to provide an appropriate definition of the thin zone, we will review the notion of 
a definable neighborhood of a compact definable subset of some Euclidean space.  
 
\medskip
Let $Z$ be a compact definable subset of some Euclidean space.
\begin{definition}\label{def:def-nghb}
A definable neighborhood of the definable compact subset $Y$ of $Z$ is any neighborhood of $Y$ in $Z$
of the form $f^{-1}([0,\ve[)$, where $f:Z\to\R_+$ is a non-negative continuous definable function
vanishing exactly on $Y$ and $\ve$ is a positive real number.
\end{definition}
Since $Z$ is embedded in some Euclidean space, the outer distance function on $Z$ to 
any definable closed subset of $Z$ defines a (definable) 
family of definable neighborhoods of the given closed subset.

Let $Y$ be a definable compact subset of $Z$ and let $\calU_\ve := f^{-1}([0,\ve[)$ be a definable 
neighborhood of $Y$. Since the definable family $(f^{-1}([0,t[))_{t>0}$ admits only finitely
many topological type \cite{vdD,Co}, we can always take $\ve$ such that for any positive real number $\ve' < \ve$,
the definable neighborhoods $\calU_\ve$ and $\calU_{\ve'}$ are (definably) homeomorphic.

The other important property of definable neighborhoods of definable compact subsets 
is that there are unique, up to a homeomorphism. More precisely 
\begin{proposition}\label{prop:defin-neighb}
Let $\calU_\ve := f^{-1}([0,\ve[)$ and $\calV_\delta:= g^{-1}([0,\delta[)$ be two definable neighborhoods of $Y$ 
a definable compact subset of $Z$ for any positive real numbers $\ve,\delta$ small enough. Then $\calU_\ve$ 
and $\calV_\delta$ are homeomorphic.    
\end{proposition}
Before sketching the proof, let us say that we can have a much more precise
statement, but what we need to define the notion of \em thin zone \em is just what is stated.  
\begin{proof}
The proof works very much like for the semi-algebraic case \cite{Dur}. The main new ingredient 
in this o-minimal setting is that we can definably stratify any continuous definable function
as a $C^r$ function on each stratum, for any a-priori prescribed positive integer $r$, such 
that Thom's condition $(a_f)$ is satisfied along each stratum \cite{Loi}. Assume that $\calU_\ve$ 
is contained in $\calV_\delta$. Let $S := g^{-1}([0,\delta]) \setminus f^{-1}([0,\ve[)$ and let us 
define, as in the proof of \cite[Lemma 1.8]{Dur}, the following continuous definable function $\kappa:S \to [0,1]$ 
\begin{center}
$\displaystyle{\kappa(x) := \frac{f(x) - \ve}{(f(x) -\ve) - (g(x) -\delta)}}$
\end{center}
The function $\kappa$ is non negative and vanishes only on $f^{-1}(\ve)$.
If $\ve$ and $\delta$ are small enough, then the gradients of the restrictions of the functions
of $f$ and $g$ to a given stratum do not point in opposite directions
(see \cite[Lemma 1.8]{Dur}). Thus the gradient of the restriction 
of the function $\kappa$ along any stratum does not vanish. We can refine the 
definable stratification so that $\kappa$ satisfies Thom's condition $(a_f)$. 
Using Thom First Isotopy Lemma we conclude that this function induces a locally trivial 
fibration over its image $[0,1]$. The trivialization is realized by the flow of
a vector field, so that the level $f^{-1}(\ve)$ is pushed continuously onto 
the level $g^{-1}(\delta)$. This provides a continuous flow $F: f^{-1}(\ve)\times [0,1] \to S$
such that $F(x,0) = x$.

If $\ve$ was taken small enough to start with, there is also a continuous flow $G : f^{-1}(\ve)\times [0,\ve[ \to 
f^{-1} (]0,\ve])$ such that $f(G(x,t)) = \ve -t$ and $G(x,0) =x$, and moreover the definable neighborhood
$f^{-1} ([0,\ve[)$ is homeomorphic to the definable neighborhood $f^{-1} ([0,\ve'[)$, for any positive real number 
real number $\ve' <\ve$.

We deduce from these two flows that the definable neighborhood $f^{-1} ([0,\ve'[)$ is homeomorphic to the  
definable neighborhood $g^{-1}([0,\delta[) = S \cup f^{-1} ([0,\ve[)$ for any positive $\ve' <\ve$.
\end{proof}
We can now present the definition of the topological thin zone since
Proposition \ref{prop:defin-neighb} (when the definable neighborhood is chosen small enough) 
guarantees it is relevant:  
\begin{definition}\label{def:thin-zone}
Let $X_1$ be a simply embedded model of a closed definable isolated singularity $X_0$ of some Euclidean 
space, thick at the origin and with connected link at the origin. A \em topological thin zone \em of $X_1$ is 
the image by the blowing-down mapping $\beta$ of a definable neighborhood in the strict transform $X_1'$ 
of the non-simple point locus $\ns(\bdr X_1') = \Sns X_1 \times 0$ of the boundary $\bdr X_1'$ of $X_1'$. 
We will denote it by $\thin (X_1)$. 
The complement of a topological thin zone is called a \em topological thick zone. \em 
\end{definition}
The standard family of definable neighborhoods we have in mind is obtained from the restriction 
to the strict transform $X_1'$ of the Euclidean distance function coming from the inclusions  
$X_1' \subset \Srn\times \R_+ \subset\Rn\times \R_+$. 
\\
Given an isomorphism $\Phi: X_1 \to X_2$ between two simply embedded models, in some Euclidean spaces $\R^p$
and $\R^q$ respectively, of a closed definable isolated singularities and thick at the origin, 
the interest of this definition is that a topological thin zone if mapped onto a topological thin zone. 
Indeed a topological thin zone of the simply embedded model $X_2$ is defined as $\beta_q(f_2^{-1}([0,\ve[))$ 
for a continuous definable function $f_2:X_2'\to \R_+$ vanishing exactly on $(\Sns X_2) \times 0\cap \bdr X_2$. 
Defining a continuous definable function $f_1 :X_1' \to \R_+$ as $f_1 := f_2 \circ \Phi'$, we see that the 
homeomorphism $\Phi'$ maps the definable neighborhood $f_1^{-1}([0,\ve[)$ onto 
the definable neighborhood $f_2^{-1}([0,\ve[)$, so that $\Phi$ maps the topological 
thin zone $\beta_p(f_1^{-1}([0,\ve[))$ onto the topological thin zone $\beta_q(f_2^{-1}([0,\ve[))$. Thus the 
decomposition into topological thick zone/topologically thin zone is preserved by isomorphisms. It is as we 
were expected an invariant decomposition for our class of isomorphisms which, as we have repeatedly said, 
contains definable bi-Lipschitz homeomorphisms for the outer metrics.

\medskip 
We can now rephrase Theorem \ref{thm:main} in terms of the topological thin zone:
\begin{theorem}\label{thm:main-bis}
Let $X_1$ be a simply embedded model of a closed definable isolated singularity of some Euclidean
space, thick at the origin and with connected link. 
The inclusion mapping induces the following exact sequence 
\begin{equation}\label{eq:main-bis-thm}
\calH_\bullet (\thin (X_1),\bo) \to \calH_\bullet (X_1,\bo) \to 0.
\vspace{4pt}
\end{equation}
\end{theorem}
Its expected version with degrees and dimension drops is the following 
\begin{corollary}\label{cor:main-cor-bis}
For any degree $k=1,\ldots,\dim X_1,$ and for any dimension drop $\delta = 1,\ldots,k,$
the inclusion mapping induces the following exact sequence 
\begin{equation}\label{eq:main-1-bis}
\calH_k^\delta (\thin(X_1),\bo) \to \calH_k^\delta (X_1,\bo) \to 0.
\vspace{4pt}
\end{equation}
\end{corollary}
%
%
%
%
%
%
%
%
%
%
%
%
%
%
%
%
%
%
%
%
%
%
%
%
%
%
%
\section{Non-simply embedded thick isolated singularities}\label{Section:SETIS}
We are going to have here a short view on what is happening for 
closed definable isolated singularities of some Euclidean space, thick at the origin and with 
connected link which are not simply embedded. 

\medskip
Any closed definable isolated singularity $X_0$ thick at the origin with connected link 
admits a simply embedded model $X_1$ and thus exists a simple embedding morphism 
$\Phi: X_1\to X_0$. As a consequence of our definitions and of Bernig \& Lytchak results \cite{BeLy},
it provides a surjective continuous definable finite-to-one mapping  $\bdr X_1' \to \bdr X_0'$. 

\smallskip
Let $\Phi_i:X_i\to X_0$, $i=1,2$, be two simple embedding morphisms of the closed definable isolated 
singularity $X_0$ thick at the origin with connected link. By definition of simply embedding morphism, 
the mapping $\Phi_2^{-1}\circ \Phi_1 : X_1 \to X_2$ is a definable homeomorphism which extends as an isomorphism 
$X_1' \to X_2'$. A topological thin zone of $X_1$ is then mapped isomorphically onto a topological thin zone of 
$X_2$. Thus we can introduce the following:
\begin{definition}\label{def:thinzone-nse}
The \em topological thin zone \em of $X_0$ is defined as the image of the topological thin zone of any simply embedded model 
of $X_0$ by the corresponding simple embedding morphism of the simply embedded model.
\end{definition}
As observed just before Definition \ref{def:thinzone-nse}, defining the topological thin zone this way 
is really relevant since it does not depend on the simply embedded model of $X_0$. 
The first consequence is the following
\begin{proposition}\label{prop:homol-thin-same}
Let $\Phi:X_1\to X_0$ be a simply embedding morphism of a closed definable isolated singularity $X_0$
thick at the origin and with connected link. For any degree $k$ and any dimension drop $\delta =1,\ldots,k$, 
the homomorphism $\Phi_*$ induces an isomorphism $\calH_k^\delta (\thin(X_1),\bo) \to \calH_k^\delta 
(\thin(X_0),\bo)$.
\end{proposition}
\begin{proof}
Since $\thin(X_0) = \Phi (\thin (X_1))$ injectivity is immediate. Since the mapping $\Phi$ extends as a morphism,
any fc-cycle supported in the thin zone $\thin(X_1)$ is mapped onto a fc-cycle supported in $\thin (X_0)$.
Since the homeomorphism $\Phi$ induces a surjective continuous definable finite-to-one mapping $\So X_1 \to \So X_0$, 
any fc-cycle supported in $\thin (X_0)$ and of dimension drop larger than or equal to $\delta$ is mapped by $\Phi^{-1}$ 
into a fc-cycle supported in $\thin (X_1)$ of dimension drop larger than or equal to $\delta$, thus we find the 
surjectivity.
\end{proof}
A straightforward corollary of Proposition \ref{prop:homol-thin-same} is 
\begin{proposition}
The inclusion mapping of the thin zone $\iota:\thin(X_0)\to X_0$ gives the following exact sequences:
\vspace{4pt}
\begin{equation}\label{eq:main-2}
\iota_*:\calH_k^\delta (\thin(X_0),\bo) \to \calH_k^\delta (X_0,\bo) \to 0
\vspace{4pt}
\end{equation}
for any degree $k$ and any dimension drop $\delta\leq k$.
\end{proposition}
\begin{proof}
Let $\Phi:X_1\to X_0$ be a simply embedded morphism of the singularity $X_0$. We obtain a homomorphism 
$\Phi_*:\calH_\bullet (X_1,\bo)\to \calH_\bullet (X_0,\bo)$. Since the mapping induced on the tangent links 
$\So X_1 \to \So X_0$ is a surjective continuous definable finite-to-one mapping, any fc-cycle of $X_0$ is 
pulled back onto a cycle of $X_1\setminus \bo$ such that it is contracted by a chain of the simply embedded model 
$X_1$ whose tangent cone at the origin is of the exact same dimension 
as the chain contracting the fc-cycle of $X_0$, then the pre-image by the mapping $\Phi$ is a fc-cycle of $X_1$. 
Thus the homomorphism $\Phi_*:\calH_k^\delta (X_1,\bo) \to \calH_k^\delta (X_0,\bo)$ is surjective for any 
degree $k$ and drop dimension $\delta =1,\ldots,k$. 
Proposition \ref{prop:homol-thin-same} allows to conclude.
\end{proof}
%
%
%
%
%
%
%
%
%
%
%
%
%
%
%
%
%
%
%
%
%
%
%
\section{Examples and applications}\label{Section:examples}
This last section presents some examples of fc-homology classes of thick singularities. First we find out 
that the three known obstructions (so far) to the local inner metric conicalness of complex singularities,
namely \em fast loops, separating sets \em and \em choking horns, \em are not of metric nature but
of something slightly less rigid. Second they illustrate a same phenomenon: they represent fc-homology classes 
which may be non trivial in the topological thin zone. Proposition \ref{prop:sep-set} characterizes 
completely the existence of separating sets of simply embedded models of closed definable thick isolated 
singularities with connected link. Last, when the tangent link is of pure dimension 
Theorem \ref{thm:sep-set-pure-dim} asserts that separating sets represents some non trivial
fc-homology classes of the topological thin zone but trivial in the link.

\medskip
We recall quickly the notions of \em fast loops, separating sets and choking horns. \em
Although these notions were introduced for normally embedded models of semi-algebraic singularities,
it is elementary to check that they are all preserved under isomorphisms since each can be defined in 
terms of properties on the link of some tangent cone.
\begin{definition}[\cite{BFN3,Fe2,BNP,BFGoS}]\label{def:fast-loop-etc}
Let $X$ be a normally embedded model of a closed definable isolated singularity of some Euclidean space
with connected link at the origin. 

- A \em fast loop \em is a continuous definable mapping $\gamma: \Sr^1 \to X$ such that 
there exists a continuous definable mapping $h: [0,1]\times \Sr^1 \to X$ such that
\begin{enumerate} 
\item $h (0,\bu) = \bo$ and $h (1,\bu) = \gamma (\bu)$,
\item the image {\em Im}($h$) of the mapping $h$ is thin.
\end{enumerate}

- A \em choking horn in $X$ \em is the image $H$ of a definable continuous mapping 
$\phi\colon [0,1]\times \Sr^p\rightarrow X$, for some positive integer $p$, such that
\begin{enumerate}
\item For any point $(t,\bu)\in[0,1]\times \Sr^p$, its image $\phi(t,\bu)$ lies in the 
intersection $X_t:=X\cap\Sr_t^{n-1}$.
\item The tangent cone of the image of $\phi$ at $\bo$ is just a single real half-line.
\item For every $t$ small enough and any chain $\eta_t$ whose support is contained in the subset $X_t$
and such that the boundary of its support $\bdr (\supp(\eta)_t)$ coincides with the 
intersection $H_t:= H\cap \Sr_t^{n-1}$, there exists a positive real constant $K$ such that 
the diameter of the support of $\eta_t$ satisfies the inequality ${\rm diam}(\supp(\eta)_t)\geq Kt$.
\end{enumerate}

- A separating set $S$ of $X$ is a closed definable subset of $X$ of codimension one, thin at the origin,
such that the germ of the complement $X\setminus S$ has at least two thick connected components.
\end{definition}
We observe that the both notions of fast loop and choking horn represent a fc-contracting homology class. 
For isolated complex singularities the 
presence of either one is an obstruction to metric conicalness \cite{BF3,BFN1,BFN3,Fe2,BFGoS}. 
Note also that the paper \cite{BNP} considers some special fast loops, called \em of the second kind, \em
which are just choking horns of dimension $p=1$.
The main technical result of \cite{BFGoS} asserts that a subanalytic cone (and thus any subanalytically 
bi-Lipschitz homeomorphic image), cannot admit choking horns. Since the proof of this result 
\cite[Theorem 2.6]{BFGoS} relies on arguments about the diameters of some family of links, which can 
be reformulated in terms of dimension of Hausdorff limits of the linearly rescaled family, we observe 
that with the exact same arguments we actually have the following more general result (which could also 
deduce from Corollary \ref{cor:cone-no-hom}):  
\begin{proposition}
Let $X_1$ be a closed definable subset isomorphic to a cone over a topological submanifold.
Then $X_1$ cannot admit a choking horn. 
\end{proposition}
The presence of a separating set from the point of view of the fast contracting homology 
is slightly more complicated to deal with and it can only occur in thick isolated singularities. 
According to \cite{BFN3} the tangent cone at the origin of a separating set separates the tangent cone 
of the ambient subset. It is a property on the tangent links. Consequently, 
as we have already suggested, a separating set is a notion invariant by isomorphisms. 

The next definition, as we will see, is key to characterize completely the notion of separating set.
\begin{definition}\label{def:cond-SC} 
Let $X_1$ be a simply embedded model of a closed definable isolated singularity of some Euclidean space,
thick at the origin and with connected link. Let us consider the following condition:\\
\em (SC): \em The locus of non-simple tangent directions $\Sns X_1$ separates the tangent link $\So X_1$
and there exists a non empty connected component of the complement $\So X_1 \setminus \Sns X_1$ such that 
its closure intersects with $\Sns X_1$ in codimension larger than of equal to two. 
\end{definition}

The main interest of condition (SC) in relation with separating sets lies in the following
\begin{lemma}\label{lem:sep-set-1}
Let $X_1$ be a simply embedded model of a closed definable isolated singularity of some Euclidean space,
thick at the origin and with connected link. If Condition $(SC)$ holds true, then $X_1$ admits a separating set.
\end{lemma}
\begin{proof}
Let $\calU$ be a connected component of the simple tangent directions locus $\Sos X_1 = \So X_1 \setminus \Sns X_1$ 
whose closure intersects with $\Sns X_1$ in codimension at least two. Let $S$ be a connected component of this intersection.
A horn neighborhood of the cone over $S$ is a subset of the form
\begin{center}
$\calU(1+a,K) = \{x\in\Rn\setminus\bo: \dou (\nu(x),S) < K |x|^{1+a}\} \cup \{\bo\}$,
\end{center} 
for a positive real number $a$ and some positive constant $K$. There exists a connected component $\mathcal{C}$ of the subset $X_1\setminus\calU(1+a,K)$ whose tangent cone is exactly $\clos(\mathcal{C})$, for small enough positive real numbers $a$ 
and for some positive constant $K$. Thus the boundary of the complement $X_1\setminus\calU(1+a,K)$ is a separating set.
\end{proof}

\smallskip
The relation of a separating set with the topological thin zone, its fast contracting homology and the
locus of non-simple tangent directions is described in the next result.
\begin{lemma}\label{lem:sep-set-2}
Let $X_1$ be a simply embedded model of a closed definable isolated singularity of some Euclidean space,
thick at the origin and with connected link. If there exists a separating set $S_1$, 
then there exists another separating set $S$ contained in the topological thin zone $\thin (X_1)$ 
such that for any positive real number $t$ small enough, the intersection $S_t = S\cap \Sr_t^{n-1}$ 
gives rise to a fc-cycle which cannot be contracted within the topological thin zone $\thin (X_1)$. 
In other words the inclusion mapping $\iota:\thin (X_1) \to X_1$ induces 
a surjective and non-injective homomorphism
\begin{center}
$\iota_* :\calH_{d-2} (\thin(X_1),\bo) \to \calH_{d-2}(X_1,\bo)$.
\end{center} 
Moreover condition \em (SC) \em holds true. 
\end{lemma}
\begin{proof}
Since the subset $S_1$ separates $X_1$, let $Y_1$ be one of the thick connected components of the 
complement $X_1 \setminus S_1$ and let $Y_2$ be the union of the other connected components. Both are 
open and thick subsets of $X_1$. The intersection of their closures $\clos (Y_1) \cap \clos (Y_2)$ is 
contained in the separating set $S_1$. We observe that the complement of this intersection 
$X \setminus (\clos (Y_1) \cap \clos (Y_2))$ coincides with the disjoint union 
$\intr (\clos (Y_1)) \sqcup \intr (\clos (Y_2))$ of the interior of the closures
of $Y_1$ and $Y_2$. Moreover the tangent link of $\clos (Y_1) \cap \clos (Y_2)$ 
cannot contain any interior point of the tangent links of $\clos(Y_1)$ and $\clos (Y)_2$.
Thus the tangent link of $\clos (Y_1) \cap \clos (Y_2)$ cannot contain any simple tangent direction,
so it is contained in $\Sns X_1$ the locus of non-simple tangent directions. 
Since the intersection $\clos (Y_1) \cap \clos (Y_2)$ separates the subset $X_1$ and is contained in $\clos(Y_1)$, 
it is a separating set. Moreover the subset $\clos (Y_1) \cap \clos (Y_2)$ is contained in the thin zone
since its tangent link is contained in $\Sns X_1$.
Last $\clos (Y_1) \cap \clos (Y_2)$ is the boundary of the disjoint union $\intr (\clos (Y_1)) \sqcup \intr (\clos (Y_2))$.
Defining $S$ as the intersection $\clos (Y_1) \cap \clos (Y_2)$, we deduce that 
the family $(S_t)_t$, intersections of the subset $S$ with the Euclidean spheres of small radius $t$, are fc-cycles 
which are trivial in the link $(X_1)_t$ but which cannot be contracted in any sufficiently small 
conic neighborhood $\cone(\Sns X_1,\ve)\cap X_1$ of the cone over the locus of non-simple 
tangent directions, that is the topological thin zone.

It is clear that the intersection of the locus $\Sns X_1$ with the tangent link of the closure
$\clos(Y_1)$ is of codimension at least two since the tangent link of the closure $\clos (Y_1)$ 
is disconnected from its complement in $\So X_1$ by the tangent link $\So S$ of the subset $S$.
\end{proof}
We can now state the first main result of this section which is just a juxtaposition of the 
two previous Lemma. 
\begin{proposition}\label{prop:sep-set}
Let $X_1$ be a simply embedded model of a closed definable isolated singularity of some Euclidean space,
thick at the origin and with connected link. 

There exists a separating set if and only if the locus of non-simple tangent directions $\Sns X_1$ 
satisfies condition \em (SC). \em 
\end{proposition}
The second main result of this section clearly highlights the connection between separating sets and 
the fast contracting homology of the topological thin zone. 
\begin{theorem}\label{thm:sep-set-pure-dim}
Let $X_1$ be a simply embedded model of a closed definable isolated singularity of some Euclidean space
thick at the origin and with connected link. Suppose the tangent link is of pure dimension $d-1$. 
There exists a separating set if and only the surjective homomorphism of the fc-homology groups 
\begin{center}
$\iota_*:\calH_{d-2} (\thin(X_1),\bo) \to \calH_{d-2}(X_1,\bo)$
\end{center} 
induced by the inclusion mapping $\iota:\thin (X_1) \to X_1$ is not injective
\end{theorem}
\begin{proof}
One of the statement is just of Lemma \ref{lem:sep-set-2}.

\smallskip
Suppose that the homomorphism is not injective. 
There exists a $(d-2)$-cycle $\xi$ contracted by a chain $\eta$ with thin support and containing the origin.
Since the cycle $\xi$ is trivial in the link $\So X_1$ it separates the link, and thus
the support of $\eta$ must separate $X_1$. Let $Y_1$ be a thick connected component
of the complement $X_1\setminus\supp (\eta)$. Assume that all the other connected components are thin.
Since the cycle $\xi$  is not trivial in the thin zone, it cannot be contracted in the link of the thin zone $\thin (X_1)$.
We deduce that there must be a point in the tangent link which cannot lie in the union of the tangent link
of the thin zone $\thin (X_1)$ with the tangent link of $\clos(Y_1)$. This implies that the tangent link of the 
closure of the complement $X_1\setminus (\supp(\eta)\cup Y_1)$ is of dimension smaller than or equal to $\dim X_1 -2$
and cannot be contained in the tangent link of the union $Y_1\cup\thin(X_1)$ (which contains the tangent link 
of the union $Y_1\cup\supp(\eta)) $. 
This latter fact contradicts that the tangent link $\So X_1$ is of pure dimension $\dim X_1 -1$.
\end{proof}
Separating sets were introduced and exhibited as objects whose presence denies the local inner metric 
conicalness of the set-germ at the considered point.

Proposition \ref{prop:sep-set} provides an equivalent formulation of the existence of
separating set for simply embedded models of thick isolated singularities by means of some topological conditions 
on the tangent link (Condition (SC)). Theorem \ref{thm:sep-set-pure-dim}
emphasizes the non-triviality in the topological thin zone of the fast contracting homological nature of 
separating sets.  

These last two results suggest that the notion of separating set could be slightly generalized in order to still 
detect an obstruction to (local) metric conicalness.

\end{document}